\documentclass[twoside,a4paper,leqno,12pt]{amsproc}
\usepackage[top=30mm,right=30mm,bottom=30mm,left=30mm]{geometry}
\usepackage{mathtools} 

\usepackage[pagebackref]{hyperref}
\hypersetup{citecolor=blue, linkcolor=blue, colorlinks=true}
\usepackage{amsmath, amssymb, amsthm, stmaryrd, amsfonts, xcolor, amsrefs}
\usepackage{graphicx, float, array, multicol, rotating, tikz, booktabs}
\usetikzlibrary{arrows}
\usetikzlibrary{shapes}
\usepackage{tablefootnote}      
\def\@fnsymbol#1{\ensuremath{\ifcase#1\or *\or \dagger\or \ddagger\or
   \mathsection\or \mathparagraph\or \|\or **\or \dagger\dagger
   \or \ddagger\ddagger \else\@ctrerr\fi}}

\newcommand{\gQ}{\mathrel{\raise-0.5pt\hbox{?}\kern-0.5pt{>}}}
\newcommand{\lQ}{\mathrel{{<}\kern-0.5pt\raise-0.5pt\hbox{?}}}
\newcommand{\mycirc}[1]{%
  \tikz[baseline=(char.base)]\node[draw,circle, inner sep=1.5pt, minimum size=4mm,
    text height=2.5mm](char){\ensuremath{#1}} ;}
\newcommand{\mymk}[1]{%
  \tikz[baseline=(char.base)]\node[anchor=south west, draw,rectangle, rounded corners, inner sep=2pt, minimum size=4mm,
    text height=2.5mm](char){\ensuremath{#1}} ;}

\renewcommand{\le}{\leqslant}
\renewcommand{\ge}{\geqslant}

\renewcommand{\geq}{\geqslant}
\newcommand{\coloneq}{\vcentcolon=}      

\usepackage[shortlabels]{enumitem}
\setlist[enumerate]{label=\rm{(\alph*)}}

\theoremstyle{definition}
\newtheorem{definition}{Definition}
\newtheorem{remark}[definition]{Remark}

\theoremstyle{plain}
\newtheorem{theorem}[definition]{Theorem}

\newtheorem{lemma}[definition]{Lemma}

\begin{document}

\author{S.\,P. Glasby}
\author{G.\,R. Paseman}
\address{\phantom{|}\kern-1cm S.\,P. Glasby, Centre for the Mathematics of Symmetry and Computation, University of Western Australia, 35 Stirling Highway, Godroo, Perth 6009, Australia.\newline \emph{E-mail address:} {\tt\texttt{Stephen.Glasby@uwa.edu.au}}\vskip3mm\noindent
  G.\,R. Paseman. Sheperd Systems, UC Berkeley, USA.\newline \emph{E-mail address:} {\tt\texttt{sheperdsystems@gmail.com}}
}

\thanks{Acknowledgements: SPG gratefully acknowledges support from the Australian Research Council (ARC) Discovery Project DP190100450. GRP thanks family.\newline
  Keywords: Maximum, sum, binomial coefficients.\newline
  2010 Math Subject Classification: 05A10, 11B65, 94B65.\hfill
 Date: \today.
}


\title[On the maximum of the weighted binomial
        sum \texorpdfstring{$2^{-r}\sum_{i=0}^r\binom{m}{i}$}{}]
      {On the maximum of the weighted\\ binomial
        sum \texorpdfstring{$2^{-r}\sum_{i=0}^r\binom{m}{i}$}{}}
      
\date{\today}

\begin{abstract}
  The weighted binomial sum $f_m(r)=2^{-r}\sum_{i=0}^r\binom{m}{i}$ arises
  in coding theory and information theory. We prove that,
  for $m\not\in\{0,3,6,9,12\}$, 
  the maximum value of $f_m(r)$ with $0\le r\le m$ occurs
  when $r=\lfloor m/3\rfloor+1$. We also show this maximum value
  is asymptotic to $\frac{3}{\sqrt{{\pi}m}}\left(\frac{3}{2}\right)^m$
  as $m\to\infty$.
\end{abstract}

\maketitle

\section{Introduction}\label{S:intro}

  Let $m$ be a non-negative integer, and let $f_m(r)$ be the function:
  \[
    f_m(r)=\frac{1}{2^r}\sum_{i=0}^r\binom{m}{i}.
  \]
  This function arises in coding theory and information
  theory e.g.~\cite{Ash}*{Theorem 4.5.3}.
  It is desirable for a linear code to have large rate
  (to communicate a lot of information) and large
  minimal distance (to correct many errors). So for
  a linear code with parameters $[n,k,d]$, one wants
  both $k/n$ and $d/n$ to be large.
  The case that $kd/n$ is large is studied in~\cite{AGP}.
  A Reed-Muller code $\textup{RM}(r,m)$ has $n=2^m$,
  $k=\sum_{i=0}^r\binom{m}{i}$ and
  $d=2^{m-r}$ by~\cite{LX}*{\S6.2}, and hence $kd/n$ equals $f_m(r)$.
  It is natural to ask which value of $r$ maximizes $f_m(r)$, and what is
  the size of the maximum value.
  
\begin{theorem}\label{T}
  Suppose that $m,r$ are integers where $0\le r\le m$.
  The maximum value of $f_m(r)=2^{-r}\sum_{i=0}^r\binom{m}{i}$
  occurs when $r=\lfloor\frac{m}{3}\rfloor+1$ provided $m\not\in\{0,3,6,9,12\}$.
\end{theorem}

We give an optimal asymptotic bound for the maximum value of~$f_m(r)$.

\begin{theorem}\label{T:bounds}
  Suppose that $m\not\in\{0,1,3,6,9,12\}$ and
  $r_0=\lfloor \frac{m}{3}\rfloor+1$.  Then
  \begin{equation}\label{E:A1}
    \frac{1}{2^{\lfloor \frac{m}{3}\rfloor}}
    \left(1-\frac{k+2}{2(r_0+1)}\right)\binom{m}{r_0}
    < f_m(r_0) < \frac{1}{2^{\lfloor \frac{m}{3}\rfloor}}\binom{m}{r_0}
  \end{equation}
  where $k\coloneq 3r_0-m\in\{1,2,3\}$. Furthermore,  
  \begin{equation}\label{E:A2}
    f_m(r_0)<\frac{3}{\sqrt{\pi m}}\left(\frac{3}{2}\right)^{m}
    \qquad\textup{and}\qquad
    \lim_{m\to\infty} f_m(r_0)\sqrt{m}\left(\frac{2}{3}\right)^{m}
      =\frac{3}{\sqrt{\pi}}.
  \end{equation}
\end{theorem}

We prove that $f_m(r)$ increases strictly if
$0\le r\le r_0\coloneq\lfloor\frac{m}{3}\rfloor+1$ and $m>12$
(see Theorem~\ref{T:lower}), and it 
decreases strictly for $r_0\le r\le m$ (see Theorem~\ref{T:upper}).
Elementary arguments in Lemma~\ref{L:2}(c) show that
$f_m(0)<f_m(1)<\cdots<f_m(r_0-1)$. More work is required to prove that
$f_m(r_0-1)<f_m(r_0)$. Determining when $f_m(r)$ decreases involves
a delicate inductive proof requiring a growing amount of precision, and
inequalities with rational functions such as
$X_i=\frac{r-i+1}{m-r+i}$, see Lemma~\ref{L:strat}.
In Section~\ref{S:Bound} we establish
bounds (and asymptotic behavior) for $f_m(r_0)$ using standard methods.

Brendan McKay~\cite{Gmo} showed, using approximations
for sufficiently large $m$, that the maximum value of $f_m(r)$ is near $m/3$.
His method may well extend to a proof of Theorem~\ref{T}. If so,
it would involve very different techniques from ours.

\section{Data, comparisons and strategies}\label{S:red}

  The values of
  $f_m(0),f_m(1),f_m(2),\dots,f_m(m-2),f_m(m-1),f_m(m)$ appear to
  increase to a maximum and then decrease. For `large' $m$ we see that
  \[ 1<\frac{m{+}1}{2}<\frac{m^2{+}m{+}2}{8}<\cdots \;?\;\cdots
  >8-\frac{m^2{+}m{+}2}{2^{m-2}}>4-\frac{m+1}{2^{m-2}}>2-\frac{1}{2^{m-1}}>1.
  \]
  Computer calculations for `large' $m$ suggest that
  a maximum value for $f_m(r)$ occurs  at $r_0=\lfloor\frac{m}{3}\rfloor+1$, see
  Table~\ref{T:data} which lists the \emph{integer part}
  $\lfloor f_m(r)\rfloor$.  Computing $f_m(r)$ exactly shows that
  for $m\in\{0,3,6,9,12\}$ the maximum occurs at $r_0-1$
  and not~$r_0$, see Table~\ref{T:r0}. The maximum happens to occur for a
  unique $r$, except for $m=1$.

\begin{table}[!ht]
  \caption{Maximum values of $\lfloor f_m(r)\rfloor$ for $0\le r\le m$
    and $m\in\{6,7,\dots,15\}$.}\label{T:data}
  \begin{tabular}{rccccccccccccccccc}\toprule
6&& 1&3&\mycirc{5}&\mymk{5}&3&1&1 &&&&&&&\\
7&& 1&4&7&\mymk{8}&6&3&1&1 &&&&&&&\\
8&& 1&4&9&\mymk{11}&10&6&3&1&1 &&&&&&\\
9&& 1&5&11&\mycirc{16}&\mymk{16}&11&7&3&1&1 &&&&&\\
10&& 1&5&14&22&\mymk{24}&19&13&7&3&1&1 &&&&\\
11&& 1&6&16&29&\mymk{35}&32&23&14&7&3&1&1 &&&\\
12&& 1&6&19&37&\mycirc{49}&\mymk{49}&39&25&14&7&3&1&1 &&\\
13&& 1&7&23&47&68&\mymk{74}&64&45&27&15&7&3&1&1 &\\
14&& 1&7&26&58&91&\mymk{108}&101&77&50&29&15&7&3&1&1 \\
15&& 1&8&30&72&121&154&\mymk{155}&128&89&54&30&15&7&3&1&1 \\ \bottomrule
\end{tabular}
\end{table}

  Determining the relative sizes of $f_m(r)$ and $f_m(r+1)$ is reduced in 
  Lemma~\ref{L:reduce} to determining the relative sizes of 
  $\sum_{i=0}^{r}\binom{m}{i}$ and $\binom{m}{r+1}$. 

  \begin{lemma}\label{L:reduce}
    Suppose that $0\le r< m$. Then
    \begin{enumerate}[{\rm (a)}]
    \item the inequality $f_m(r)<f_m(r+1)$ is equivalent to
      $\sum_{i=0}^{r}\binom{m}{i}<\binom{m}{r+1}$,
    \item if $\sum_{i=0}^{r}\binom{m}{i}\le\binom{m}{r+1}$, then
      $\sum_{i=0}^{r}\binom{m+1}{i}<\binom{m+1}{r+1}$,
    \item the inequality $f_m(r)>f_m(r+1)$ is equivalent to
      $\sum_{i=0}^{r}\binom{m}{i}>\binom{m}{r+1}$, and
    \item if $\sum_{i=0}^{r}\binom{m}{i}\ge\binom{m}{r+1}$, then
      $\sum_{i=0}^{r}\binom{m-1}{i}>\binom{m-1}{r+1}$.
    \end{enumerate}
  \end{lemma}

  \begin{proof}
    (a,b)~Clearly $f_m(r)<f_m(r+1)$ is equivalent to
    $2\sum_{i=0}^{r}\binom{m}{i}<\sum_{i=0}^{r+1}\binom{m}{i}$ which is
    equivalent to  $\sum_{i=0}^{r}\binom{m}{i}<\binom{m}{r+1}$.
    If $r<m$ and $\sum_{i=0}^{r}\binom{m}{i}\le\binom{m}{r+1}$,~then
    \begin{align*}
    \sum_{i=0}^{r}\binom{m+1}{i}
    &=\sum_{i=0}^{r}\frac{m+1}{m-i+1}\binom{m}{i}
    \le\frac{m+1}{m-r+1}\sum_{i=0}^{r}\binom{m}{i}\\
    &<\frac{m+1}{m-r}\binom{m}{r+1}
    =\binom{m+1}{r+1}.
    \end{align*}

    (c,d)~Clearly $f_m(r)>f_m(r+1)$ is equivalent to
    $2\sum_{i=0}^{r}\binom{m}{i}>\sum_{i=0}^{r+1}\binom{m}{i}$ which, in turn, is
    equivalent to  $\sum_{i=0}^{r}\binom{m}{i}>\binom{m}{r+1}$.
    If $\sum_{i=0}^{r}\binom{m}{i}\ge\binom{m}{r+1}$, then as $m>r\ge0$,
    \begin{align*}
    \sum_{i=0}^{r}\binom{m-1}{i}
    &=\sum_{i=0}^{r}\frac{m-i}{m}\binom{m}{i}
    \ge\frac{m-r}{m}\sum_{i=0}^{r}\binom{m}{i}\\
    &>\frac{m-r-1}{m}\binom{m}{r+1}
    =\binom{m-1}{r+1}.\qedhere
    \end{align*}
\end{proof}

  The following easy lemma elucidates which $r\in\{0,\dots,m\}$
  maximize $f_m(r)$.

\begin{lemma}\label{L:2}
  Let $s_m(m+1) = 2^m$, and for $0\le r\le m$ define
  \[
  s_m(r)=\sum_{i=0}^r\binom{m}{i},\qquad t_m(r)=\frac{s_m(r+1)}{s_m(r)},
  \quad\textup{and}\quad
  c_m(r)=\frac{\binom{m}{r+1}}{\binom{m}{r}}=\frac{m-r}{r+1}.
  \]
  {\rm(a)}~If $0\le r\le m$, then $c_m(r)<t_m(r)$, and if $0\le r<m$, then
  $t_m(r+1)<t_m(r)$.\newline
  {\rm(b)}~If $m\ge2$, then  for some $r^*$,
  $f_m(0)<\cdots<f_m(r^*)$ and $f_m(r^*+1)>\cdots>f_m(m)$.\newline
  {\rm(c)}~$\max\{f_m(0),\dots,f_m(m)\}=\max\{f_m(r^*),f_m(r^*{+}1)\}$
    and $f_m(0)<\cdots<f_m(r_0{-}1)$.
\end{lemma}

\begin{proof}
  (a)~We show $c_m(r)<t_m(r)$ via induction on $r$. This is true when $r=0$ as
  $c_m(0)=m<m+1=t_m(0)$.  Suppose that $0\le r<m$ and $c_m(r)<t_m(r)$ holds. That is,
  $\binom{m}{r+1}/\binom{m}{r}<s_m(r+1)/s_m(r)$ holds. Since
  $c_m(r+1)=\frac{m-r-1}{r+2}<\frac{m-r}{r+1}=c_m(r)$ we have
  $c_m(r+1)<c_m(r)<t_m(r)$. Using properties of mediants, it follows that
  \[
  c_m(r+1)=\frac{\binom{m}{r+2}}{\binom{m}{r+1}}
  <\frac{\binom{m}{r+2}+s_m(r+1)}{\binom{m}{r+1}+s_m(r)}<\frac{s_m(r+1)}{s_m(r)}=t_m(r).
  \]
  Hence $c_m(r+1)<t_m(r+1)<t_m(r)$ as $s_m(n+1)=\binom{m}{n+1}+s_m(n)$.
  This completes the induction, and it also proves that $t_m(r+1)<t_m(r)$,
  as claimed.
  
  (b)~Since $s_m(m+1) = 2^m$, part~(a) shows that
  $1=t_m(m)<\cdots<t_m(0)=m+1$. Choose an integer $r^*$ such that
  $t_m(r^*)\le 2<t_m(r^*-1)$. The following are equivalent:
  $2<t_m(r)$; $2s_m(r)<s_m(r+1)$; $f_m(r)<f_m(r+1)$.
  Thus $2<t_m(r^*-1)<\cdots<t_m(0)$ implies
  $f_m(0)<\cdots<f_m(r^*)$.
  Similarly, $t_m(m-1)<\cdots<t_m(r^*+1)<2$
  and $t_m(r)<2$ implies $f_m(r+1)<f_m(r)$.
  Hence $f_m(r^*+1)>\cdots>f_m(m)$.

  (c)~By part~(b), 
  $\max\{f_m(r)\mid 0\le r\le m\}=\max\{f_m(r^*),f_m(r^*+1)\}$.
  If $2\le c_m(r)=\frac{m-r}{r+1}$, then $3r+2\le m$ and
  $r\le\lfloor\frac{m-2}{3}\rfloor$. Hence $2\le c_m(r)<t_m(r)$
  by part~(a), and $\lfloor\frac{m-2}{3}\rfloor\le r^*-1$ by
  the definition of $r^*$. Thus $r_0-1=\lfloor\frac{m}{3}\rfloor\le r^*$ and
  it follows from part~(b) that $f_m(0)<\cdots<f_m(r_0-1)$.
%
\end{proof}


  Fix $m$ and $r$ where $0\le r< m$.
  We shall use the following notation:
    \begin{align}
      X_i&=\frac{r-i+1}{m-r+i} &\textup{for $0\le i\le r$},\label{E:X}\\
      S_j&= 1+X_{j+1}+X_{j+1}X_{j+2}+\cdots+X_{j+1}X_{j+2}\cdots X_r
          &\textup{for $0\le j< r$,} \label{E:S}\\
      T_{j}&=1+X_1+X_1X_2+\cdots+X_1X_2\cdots X_j
          &\textup{for $0\le j\le r$.} \label{E:T}
    \end{align}
    Our convention in~\eqref{E:T} is that $T_0=1$ as
    $T_j=\sum_{i=0}^j(\prod_{k=1}^i X_k)$ equals 1 when $j=0$.

\begin{lemma}\label{L:strat}
  Fix $m,r,j$ where $0\le j\le r< m$. Using the above definitions, 
\begin{enumerate}[{\rm (a)}]
\item the inequality $\sum_{i=0}^j\binom{m}{r-i}>\binom{m}{r+1}$
  is equivalent to $T_j>X_0^{-1}$,
\item the inequality $\sum_{i=0}^r\binom{m}{i}<\binom{m}{r+1}$
  is equivalent to $S_0<X_0^{-1}$.
\end{enumerate}
\end{lemma}

\begin{proof}
For $0\le i\le r$, we have $\binom{m}{r-i}=X_i\binom{m}{r-i+1}$ so
$\binom{m}{r-i}=(\prod_{k=1}^i X_k)\binom{m}{r}$ holds.
Therefore $\sum_{i=0}^{j}\binom{m}{r-i}
  =\binom{m}{r}\sum_{i=0}^{j}\left(\prod_{k=1}^i X_k\right)
  =\binom{m}{r}T_j$.
Since $\binom{m}{r}=X_0\binom{m}{r+1}$, the inequality
$\sum_{i=0}^{j}\binom{m}{r-i}>\binom{m}{r+1}$ is equivalent to
$\binom{m}{r}T_j>X_0^{-1}\binom{m}{r}$ which is equivalent to
$T_j>X_0^{-1}$. This proves part~(a).

Note that 
$\sum_{i=0}^{r}\binom{m}{i}=\sum_{i=0}^{r}\binom{m}{r-i}
=\binom{m}{r}T_r=\binom{m}{r}S_0$ since $S_0=T_r$.
Since $\binom{m}{r+1}=X_0^{-1}\binom{m}{r}$, the inequality
$\sum_{i=0}^{r}\binom{m}{i}<\binom{m}{r+1}$ is equivalent to
$\binom{m}{r}S_0<X_0^{-1}\binom{m}{r}$ which is equivalent to
$S_0<X_0^{-1}$. This proves part~(b).
\end{proof}

\section{Proof that \texorpdfstring{$f_m(r)$}{} is increasing
  for \texorpdfstring{$0\le r\le r_0$}{}}\label{S:inc}


Recall that $m\ge0$ and $r_0\coloneq\lfloor\frac{m}{3}\rfloor+1$.
We now strengthen Lemma~\ref{L:2}(c).




\begin{theorem}\label{T:lower}
  If $m\not\in\{0,1,3,6,9,12\}$, then $f_m(0)<f_m(1)<\cdots<f_m(r_0)$.
\end{theorem}

\begin{proof}
  The statement is easy to check for $m\in\{2,4,5\}$.
  The statement follows from Tables~\ref{T:data} and~\ref{T:r0}
  for $m\in\{7,8,10,11,13,14\}$. Suppose now that $m\ge15$.
  
  \begin{table}[!ht]
  \caption{$\square=\max\{f_m(r_0-1), f_m(r_0)\}$ for $0\le m\le12$,
    $r_0=\lfloor m/3\rfloor+1$.}\label{T:r0}
  \begin{tabular}{rcccccccccccccc}
    \toprule
    $m$&&0&1&2&3&4&5&6&7&8&9&10&11&12\\  [0.5mm]
    $f_m(r_0-1)$&&\fbox{$1$}&\fbox{$1$}&$1$&\fbox{$2$}&$\frac52$&$3$&\fbox{$\frac{11}2$}&$\frac{29}{4}$
      &$\frac{37}{4}$&\fbox{$\frac{65}{4}$}&$22$&$29$&\fbox{$\frac{397}{8}$}\\ [2mm]
    $f_m(r_0)$&&$\frac{1}{2}$\tablefootnote{Observe that $f_0(1)=2^{-1}\left(\binom{0}{0}{+}\binom{0}{1}\right)=2^{-1}(1{+}0)=\frac{1}{2}$.}&\fbox{$1$}&\fbox{$\frac{3}{2}$}&$\frac{7}{4}$&\fbox{$\frac{11}{4}$}
      &\fbox{$4$}&$\frac{21}{4}$&\fbox{$8$}&\fbox{$\frac{93}{8}$}&$16$&\fbox{$\frac{193}{8}$}
    &\fbox{$\frac{281}{8}$}&$\frac{793}{16}$\\[2mm]
    \bottomrule
  \end{tabular}
  \end{table}


  Recall that $r_0=\lfloor\frac{m}{3}\rfloor+1$ and
  $m\in\{3r_0-3,3r_0-2,3r_0-1\}$. By Lemma~\ref{L:2}(c) it suffices
  to show that $f_m(r_0-1)<f_m(r_0)$.
  If we prove this for $m = 3r_0 - 3$,  Lemma~\ref{L:reduce}(b,a)
  gives it for $m=3r_0-2$ and $m=3r_0-1$ as well, so for $r_0 \geq 6$
  we want to show $f_{3r_0-3}(r_0-1) < f_{3r_0-3}(r_0)$.
  This is true for $r_0=6$ by
  Table~\ref{T:data}.  We set $t\coloneq r_0-1$, $m\coloneq 3t$ and we prove, using induction
  on $t$, that $f_{3t}(t)<f_{3t}(t+1)$ holds for all~$t\ge6$.

  Note that $f_{3t}(t)<f_{3t}(t+1)$ is equivalent by Lemma~\ref{L:reduce}(a) 
  to $\sum_{i=0}^{t}\binom{3t}{i}<\binom{3t}{t+1}$, and this is equivalent to
  $S_0<X_0^{-1}$ by Lemma~\ref{L:strat}(b). Putting $m=3t$ and $r=t$
  in~\eqref{E:X}, gives  $X_i=\frac{t-i+1}{2t+i}$ and
  $S_0=1+X_1+X_1X_2+\cdots+X_1X_2\cdots X_{t}$ by \eqref{E:S}.
  
  It follows from $0<X_{t}<\cdots<X_5<X_4$ and
  $X_4=\frac{t-3}{2t+4}<\frac12$ that
  \[
    S_3=1+X_4+X_4X_5+\cdots+X_4X_5\cdots X_{t}
    <1+\frac12+\frac14+\cdots+\frac{1}{2^{t-3}}<\sum_{i=0}^\infty\frac{1}{2^i}=2.
  \]
  The recurrence relation $S_j=1+X_{j+1}S_{j+1}$ for $0\le j<t$ implies that
  \begin{align*}
    S_0&=1+X_1\left(1+X_2\left(1+X_3S_3\right)\right)
    <1+X_1\left(1+X_2\left(1+2X_3\right)\right)\\
    &= 1+\frac{t}{2t+1}\left(1+\frac{t-1}{2t+2}\left(1+\frac{2(t-2)}{2t+3}\right)\right).
  \end{align*}
  We aim to show that $S_0<X_0^{-1}$. It suffices to
  prove $1+X_1\left(1+X_2\left(1+2X_3\right)\right)\le X_0^{-1}$
  where $X_0=\frac{t+1}{2t}$.
  This amounts to proving that
  \[
    1+\frac{t}{2t+1}\left(1+\frac{t-1}{2t+2}\left(1+\frac{2(t-2)}{2t+3}
   \right)\right) \le \frac{2t}{t+1}.
  \]
  Rearranging, and using the denominator $(2t+1)(2t+2)(2t+3)$, gives
  \[
  0\le\frac{3t^2-17t-6}{(2t+1)(2t+2)(2t+3)}
  =\frac{(3t+1)(t-6)}{(2t+1)(2t+2)(2t+3)}
  \]
  This inequality is valid for all $t\ge6$.  This completes the proof.
\end{proof}

How might one prove a nice formula such as
$\lim_{s\to\infty}\sum_{i=0}^s\binom{3s}{i}/\binom{3s}{s}=2$?

\begin{remark}\label{R:3s+9}
  For $s>4$ set $m=3s$ and $r_0=s+1$. Then
  $f_m(r_0-1)<f_m(r_0)$ by Theorem~\ref{T:lower}. Hence
  $\sum_{i=0}^{s}\binom{3s}{i}<\binom{3s}{s+1}=\frac{2s}{s+1}\binom{3s}{s}$
  and so
  $\lim_{s\to\infty}\sum_{i=0}^s\binom{3s}{i}/\binom{3s}{s}\le 2$.
  We show $f_m(r_0)>f_m(r_0+1)$ in Section~\ref{S:dec}, and
  therefore $\lim_{s\to\infty}\sum_{i=0}^s\binom{3s}{i}/\binom{3s}{s}\ge 2$.
\end{remark}

\section{Proof that \texorpdfstring{$f_m(r)$}{} is decreasing
  for \texorpdfstring{$r_0\le r\le m$}{}}\label{S:dec}

Showing that $f_m(r)$ decreases strictly for $r_0\le r\le m$ is much harder.
Recall that $\binom{r}{i}=0$ if $i<0$, and
$\binom{r}{i}=\frac{1}{i!}\prod_{j=0}^{i-1}(r-j)$ if $i\ge0$.
In this section we prove: 

\begin{theorem}\label{T:upper}
  If $m\ge2$, then
  $f_m(\lfloor m/3\rfloor{+}1)>f_m(\lfloor m/3\rfloor{+}2)>{\cdots}>f_m(m)\,{=}\,1$.
\end{theorem}

Our proof of Theorem~\ref{T:upper} depends on two technical lemmas, the first
of which proves that the non-leading coefficients of a certain
polynomial $A(r)$ are all negative.


First define $B_i(r)=\prod_{\ell=1}^i(r-\ell)$. 
Now $\prod_{\ell=1}^i(r-\ell)=r^i+\sum_{k=0}^{i-1} b_{i,k}r^k$
and the coefficients $b_{i,k}$ alternate in sign: for $0\le k\le i$, we have
$b_{i,k}>0$ if $i-k$ is even and $b_{i,k}<0$ if $i-k$ is odd.
Next define polynomials $A_i(r)$ via:
\begin{equation}\label{E:A}
  A_2(r)=r^2-15r-10\quad
  \textup{and $\quad A_i(r)=(2r+i)A_{i-1}(r)-B_i(r)\quad$ for $i\ge3$.}
\end{equation}
Clearly $\deg(A_i(r))=i$ and we may write $A_i(r)=r^i+\sum_{k=0}^{i-1} a_{i,k}r^k$.
We use $a_{i,i}=1$.

Comparing coefficients in this recurrence and $B_i(r)=(r-i)B_{i-1}(r)$,
shows that
  \begin{align}
    &a_{2,0}=-10, &a_{2,1}=-15,\qquad  &a_{i,k}= ia_{i-1,k}+2a_{i-1,k-1}-b_{i,k}\quad
    &\textup{for $i\ge3$,} \tag{Ra}\label{Ra}\\
    &b_{2,0}=2, &b_{2,1}=-3,\hskip3.5mm\quad  &b_{i,k}=-ib_{i-1,k}+b_{i-1,k-1} &\textup{for $i\ge3$}.\tag{Rb}\label{Rb}
  \end{align}

\begin{lemma}\label{L3}
  Let $a_{i,k}, b_{i,k}, A_i(r), B_i(r)$ be as above.
  \begin{enumerate}
  \item If $i\ge2$, then $b_{i,i-1} = -\binom{i+1}{2}$
     and $a_{i,i-1} = -\binom{i+4}{2}$.
  \item   If $i\ge2$ and $0\le k\le i-1$, then
    $a_{i,k}\le -2b_{i,k}<0$ if $i-k$ is even,
    and $\qquad a_{i,k}\le b_{i,k}<0$ if $i-k$ is odd.
  \item If $i\ge2$, then the coefficients $a_{i,k}$ are negative for $0\le k<i$.
  \end{enumerate}
\end{lemma}

\begin{proof}
  (a)~Clearly $b_{i,i-1}=-\sum_{j=1}^i j=-\binom{i+1}{2}$. The formula for
  $a_{i,i-1}$ holds for $i=2$ and by induction using the recurrence~\eqref{Ra}.
  
  (b) We use induction on $i$. For the base case $i=2$, 
  either $i-k$ is even and $a_{2,0}=-10\le -2b_{2,0}=-4$, or $i-k$ is odd and
  $a_{2,1}=-15<b_{2,1}=-3$. Thus the claims are true for $i=2$.
  Suppose now that $i\ge3$, and the claims are valid for $i-1$.

  By part~(a), $a_{i-1,i}=-\binom{i+4}{2}<-\binom{i+2}{2}=b_{i,i-1}<0$ as claimed.
  It remains to consider $k$ in the range $0\le k<i-1$.
  It is useful to set $a_{i,-1}=b_{i,-1}=0$. Suppose first that $i-k$ is even.
  Using the recurrences~\eqref{Ra},~\eqref{Rb} and induction gives 
  \begin{align*}
    a_{i,k}&=i(a_{i-1,k}+b_{i-1,k})+(2a_{i-1,k-1}-b_{i-1,k-1})\\
    &\le i(b_{i-1,k}+b_{i-1,k})+(-4b_{i-1,k-1}-b_{i-1,k-1})\\
    &\le 2ib_{i-1,k}-2b_{i-1,k-1}=-2b_{i,k}<0.
  \end{align*}

  If $i-k$ is odd, then a similar argument gives
  \begin{align*}
    a_{i,k}&=i(a_{i-1,k}+b_{i-1,k})+(2a_{i-1,k-1}-b_{i-1,k-1})\\
         &\le i(-2b_{i-1,k}+b_{i-1,k})+(2b_{i-1,k-1}-b_{i-1,k-1})\\
    &\le -ib_{i-1,k}+b_{i-1,k-1}=b_{i,k}<0.
  \end{align*}

  (c)~This follows immediately from part~(b).
\end{proof}

\begin{lemma}\label{L2}
  Suppose that $j\ge4$. Then
  $\sum_{i=r-j}^r\binom{3r-1}{i}>\binom{3r-1}{r+1}$ holds for all
  $r$ in the range $j\le r\le\binom{j+2}{2}$.
\end{lemma}

\begin{proof}

  We apply Lemma~\ref{L:strat}(a) with $m=3r-1$.
  Hence $X_i=\frac{r-i+1}{2r+i-1}$ by~\eqref{E:X}.
  Since $\sum_{i=r-j}^r\binom{m}{i}=\sum_{i=0}^{j}\binom{m}{r-i}$ it suffices
  by Lemma~\ref{L:strat}(a) to prove that
  \[ 
     T_j=1+X_{1}+X_{1}X_{2}+\cdots+X_{1}X_{2}\cdots X_j>X_0^{-1}.
  \] 
  We prove that this inequality holds for all $r$
  in the range $j\le r\le\binom{j+2}{2}$. This inequality is equivalent to
  \begin{equation}\label{E:EE}
     X_j>X_{j-1}^{-1}(\cdots(X_2^{-1}(X_1^{-1}(X_0^{-1}-1)-1)-1)\cdots)-1.
  \end{equation}

  The right-side of~\eqref{E:EE} is a rational function in $r$, 
  which when $j=4$, equals
  \[
    \frac{P_{4}(r)}{Q_{4}(r)}=\frac{2r+2}{r-2}\left(\frac{2r+1}{r-1}\left(\frac{2r}{r}\left(\frac{2r-1}{r+1}-1\right)-1\right)-1\right)-1
  \]
  where the denominator is $Q_4(r)=(r+1)r(r-1)(r-2)$,
  and the numerator is $P_4(r)=(r+1)r(r^2-15r-10)$. Since
  $\gcd(P_4(r),Q_4(r))=(r+1)r$, the polynomials
  $A_2(r)\coloneq r^2-15r-10$ and $B_2(r)\coloneq (r-1)(r-2)$ are coprime.
  The putative inequality~\eqref{E:EE} when $j=4$ is therefore
  \[
  \frac{r-3}{2r+3}>\frac{P_{4}(r)}{Q_{4}(r)}=\frac{A_{2}(r)}{B_{2}(r)}
  =\frac{r^2-15r-10}{(r-1)(r-2)}.
  \]
  Observe that $A_2(r)<r^2-15r\le0$ for $4\le r\le15=\binom{6}{2}$.
  Thus for~$r$ in the range
  $4\le r\le \binom{6}{2}$, the left side of~\eqref{E:EE} is positive,
  and the right side is at most 0. Thus the inequality
  is valid for $4\le r\le \binom{6}{2}$ and the claim is true for $j=4$.

  Assume now that $j>4$, and that the claim is true for $j-1$.
  Therefore the inequality~\eqref{E:EE} can be written
  \[
  X_j=\frac{r-j+1}{2r+j-1}>\frac{P_{j}(r)}{Q_{j}(r)}\qquad\textup{where}\qquad
  \frac{P_{j}(r)}{Q_{j}(r)}=(X_{j-1})^{-1}\frac{P_{j-1}(r)}{Q_{j-1}(r)}-1.
  \]
  Since $(X_{j-1})^{-1}=\frac{2r+j-2}{r-j+2}$, this gives rise to the recurrences:
  \begin{align*}
     P_{j}(r)&=(2r+j-2)P_{j-1}(r)-(r-j+2)Q_{j-1}(r)&&\textup{for $j>4$,}\\
     Q_{j}(r)&=(r-j+2)Q_{j-1}(r)                  &&\textup{for $j>4$.}
  \end{align*}
  It is clear that $Q_j(r)=(r+1)r(r-1)\cdots(r-j+2)=(r+1)rB_{j-2}(r)$ holds
  and $B_{j-2}(r)$ has degree $j-2$.
  Furthermore, $(r+1)r$ divides $\gcd(P_j(r),Q_j(r))$, so the
  polynomials $A_{j-2}(r)$, which are defined by the similar
  recurrence~\eqref{E:A},  satisfy
  $P_j(r)=(r+1)r A_{j-2}(r)$ and also have degree $j-2$.

    By Lemma~\ref{L3}, $A_i(r)-r^i$ has negative coefficients and leading
    coefficient $-\binom{i+4}{2}$. So for $i\ge2$ and
    $r\le\binom{i+4}{2}$, we have $A_i(r)<r^i-\binom{i+4}{2}r^{i-1}\le 0$.
    Further, $B_i(r)=\prod_{\ell=1}^i(r-\ell)>0$ for
    $r\ge i+1$. Hence $A_i(r)/B_i(r)<0$ for $r$ satisfying
    $i+2\le r\le \binom{i+4}{2}$. Suppose that $j=i+2$, then
    $P_j(r)/Q_j(r)<0$ for $r$ in the interval
    $j\le r\le \binom{j+2}{2}$. Using the definitions of
    $P_j(r), Q_j(r)$, the inequality~\eqref{E:EE} is the same as
    \[
    X_j=\frac{r-j+1}{2r+j-1}>\frac{P_j(r)}{Q_j(r)}=\frac{A_{j-2}(r)}{B_{j-2}(r)}.
    \]
    Thus for $r$ satisfying $j\le r\le \binom{j+2}{2}$, the
    left side of~\eqref{E:EE} is positive, and the right side
    is negative. Thus the claim
    is valid for $j\le r\le \binom{j+2}{2}$.
\end{proof}


\begin{proof}[Proof of Theorem~\ref{T:upper}]
  It follows from $\sum_{i=0}^m\binom{m}{i}=2^m$ that $f_m(m)=1$.
  Since $r_0\coloneq\lfloor m/3\rfloor+1$, we have $m\in\{3r_0-3,3r_0-2,3r_0-1\}$.
  If we can prove that $f_m(r_0)>f_m(r_0+1)$ for $m=3r_0-1$, then
  $f_m(r_0)>f_m(r_0+1)$ holds for $m=3r_0-2$ and $3r_0-3$ by
  Lemma~\ref{L:reduce}(d).  With the notation in Lemma~\ref{L:2},
  we have $2>t_m(r_0)$ and hence
  $r^*\le r_0$. Therefore $f_m(r_0+1)>\cdots> f_m(m)$ holds by Lemma~\ref{L:2}(b).


  In summary, it remains to prove
  $\sum_{i=0}^{r_0}\binom{3r_0-1}{i}>\binom{3r_0-1}{r_0+1}$ for $r_0\ge1$.
  This is true for $r_0=1, 2, 3$ since $\frac{3}{2}> 1$, $4> \frac{13}{4}$
  and $\frac{93}{8}> \frac{163}{16}$.
  For each $r_0\ge4$ set $j=r_0$. Then $j\ge4$ and
  $\sum_{i=0}^{r_0}\binom{m}{i}>\binom{m}{r_0+1}$ follows by Lemma~\ref{L2}.
  This completes the proof.
\end{proof}

\begin{proof}[Proof of Theorem~\ref{T}]
  The result follows from Theorems~\ref{T:lower} and~\ref{T:upper}.
  There are two equal sized maxima if $m=1$, otherwise the maximum is unique.
\end{proof}

\section{Estimating \texorpdfstring{$f_m(r_0)$}{}}\label{S:Bound}

This section is devoted to proving asymptotically optimal bounds for $f_m(r_0)$.

\begin{proof}[Proof of Theorem~\ref{T:bounds}]
  We first prove the upper bound in~\eqref{E:A1}. This is true if $m=1$.
  For $m\not\in\{0,1,3,6,9,12\}$ and $r_0=\lfloor m/3\rfloor+1$ it follows
  from Theorem~\ref{T:lower} that $f_m(r_0-1)< f_m(r_0)$ and by
  Lemma~\ref{L:reduce}(a) that $\sum_{i=0}^{r_0-1}\binom{m}{i}<\binom{m}{r_0}$.
  Therefore $\sum_{i=0}^{r_0}\binom{m}{i}<2\binom{m}{r_0}$ and the
  upper bound follows. For the lower bound,
  $f_m(r_0)> f_m(r_0+1)$ holds by Theorem~\ref{T:upper},  and so
  $\sum_{i=0}^{r_0}\binom{m}{i}>\binom{m}{r_0+1}$ by Lemma~\ref{L:reduce}(c).
  Hence $2^{-r_0}\binom{m}{r_0+1}< f_m(r_0)$, and the lower
  bound of~\eqref{E:A1} follows from
  $\binom{m}{r_0+1}=\frac{2r_0-k}{r_0+1}\binom{m}{r_0}
  =(2-\frac{k+2}{r_0+1})\binom{m}{r_0}$.
  To prove~\eqref{E:A2}, we use  binomial approximations.

  Suppose that $0<p<1$ and $q\coloneq 1-p$. If $pn$ is an integer, then
  $qn=n-pn$ is an integer, and Stirling's approximation
  $n!=\sqrt{2\pi n}(\frac{n}{e})^n(1+\textup{O}(\frac{1}{n}))$ gives
  \begin{equation}\label{E:Stirling}
    \binom{n}{pn}=\frac{c^n}{\sqrt{2\pi pqn}}
    \left(1+\textup{O}\left(\frac{1}{n}\right)\right)
    \qquad\textup{where $c=\frac{1}{p^pq^q}$.}    
  \end{equation}
  Paraphrasing~\cite{Ash}*{Lemma~4.7.1} gives the following
  upper and lower bounds:
  \begin{equation}\label{E:Ash}
    \frac{c^n}{\sqrt{8pqn}}
    \le\binom{n}{pn}\le\frac{c^n}{\sqrt{2\pi pqn}}
    \qquad\qquad\textup{where $c=\frac{1}{p^pq^q}$.}
  \end{equation}
  
  Henceforth set $p=\frac{1}{3}$, so $q=\frac23$ and
  $c=\frac{3}{2^{2/3}}$. Therefore $c^3=\frac{27}{4}$ and
  \[
  \frac{c^{3r_0}}{2^{r_0}}=\frac{1}{2^{r_0}}\left(\frac{27}{4}\right)^{r_0}
  =\left(\frac{27}{8}\right)^{r_0}=\left(\frac{3}{2}\right)^{3r_0}
  \qquad\textup{and}\qquad
  \frac{1}{\sqrt{2pq}}=\frac{3}{2}.
  \]
  We write $m=3r_0-k$ where $k\in\{1,2,3\}$.
  
  We now prove the upper bound for $f_m(r_0)$ in~\eqref{E:A2}. It follows from
  \[
  \binom{m}{r_0}=\binom{3r_0-k}{r_0}=\frac{2r_0-k+1}{3r_0-k+1}\binom{3r_0-k+1}{r_0}
  \le\frac{2}{3}\binom{3r_0-k+1}{r_0}
  \]
  that $\binom{m}{r_0}\le(\frac{2}{3})^k\binom{3r_0}{r_0}$.
  Setting $n=3r_0$ and $p=\frac13$ in~\eqref{E:Ash} and using $m<n$ shows
  \begin{align*}
    \frac{2}{2^{r_0}}\binom{3r_0}{r_0}=\frac{2}{2^{r_0}}\binom{n}{pn}
    \le\frac{2}{2^{r_0}}\frac{c^{3r_0}}{\sqrt{2\pi pqn}}
    =\frac{2}{\sqrt{2\pi pqn}}\left(\frac{3}{2}\right)^{3r_0}
    <\frac{3}{\sqrt{\pi m}}\left(\frac{3}{2}\right)^{3r_0}.
  \end{align*}
  Using $\binom{m}{r_0}\le(\frac{3}{2})^{-k}\binom{3r_0}{r_0}$
  and $m=3r_0-k$ gives
  \[
  f_m(r_0)
  \le \frac{2}{2^{r_0}}\binom{m}{r_0}
  \le \frac{2}{2^{r_0}}\left(\frac{3}{2}\right)^{-k}\binom{3r_0}{r_0}
  < \frac{3}{\sqrt{\pi m}}\left(\frac{3}{2}\right)^{m}.
  \]
  
  We now consider approximate lower bounds for $f_m(r_0)$. Our argument
  involves constants depending on $k$ but not $r_0$ whose values
  are not relevant here. We have 
  \[
  \binom{m}{r_0+1}
  =2\left(1+\textup{O}\left(\frac{1}{r_0}\right)\right)\binom{m}{r_0}
  =4\left(1+\textup{O}\left(\frac{1}{r_0}\right)\right)\binom{m}{r_0-1}.
  \]
  Further, if $k=1,2$ and $r_0\ge1$ it follows that
  \begin{align*}
  \binom{m}{r_0-1}&=\binom{3r_0-k}{r_0-1}=\frac{3r_0-k}{2r_0-k+1}\binom{3r_0-k-1}{r_0-1}\\
  &=\left(\frac{3}{2}+\frac{k-3}{2(2r_0-k+1)}\right)\binom{3r_0-k-1}{r_0-1}
  >\frac32\binom{3r_0-k-1}{r_0-1}.
  \end{align*}
  Hence 
  $\binom{m}{r_0-1} \ge(\frac{3}{2})^{3-k}\binom{3r_0-3}{r_0-1}$
  holds for $k\in\{1,2,3\}$ and $r_0\ge1$.

  Setting $n=3r_0-3$ and $p=\frac13$ in~\eqref{E:Stirling} yields
  \begin{align*}
    \frac{1}{2^{r_0}}\binom{3r_0-3}{r_0-1}
    &=\frac{1}{2^{r_0}}\binom{n}{pn}
    =\frac{1}{2^{r_0}}\frac{c^{3r_0-3}}{\sqrt{2pq\pi n}}\left(1+\textup{O}\left(\frac{1}{n}\right)\right).
  \end{align*}
  However, $\frac{c^{3r_0}}{2^{r_0}}=(\frac{3}{2})^{3r_0}$ and
  $\frac{c^{-3}}{\sqrt{2pq\pi n}}=\frac{3c^{-3}}{2\sqrt{\pi n}}
  =\frac{2}{9\sqrt{\pi n}}=\frac{2}{9\sqrt{\pi m}}(1+\textup{O}\left(\frac{1}{m}\right))$. Therefore
  \[
    \frac{1}{2^{r_0}}\binom{3r_0-3}{r_0-1}
    =\frac{2}{9\sqrt{\pi m}}\left(\frac{3}{2}\right)^{3r_0}
    \left(1+\textup{O}\left(\frac{1}{m}\right)\right).
  \]
  The above bounds give
  \begin{align*}
  f_m(r_0)&\ge\frac{1}{2^{r_0}}\binom{m}{r_0+1}
  \ge4\left(1+\textup{O}\left(\frac{1}{r_0}\right)\right)\left(\frac{3}{2}\right)^{3-k}\frac{1}{2^{r_0}}\binom{3r_0-3}{r_0-1}\\
  &=4\left(1{+}\textup{O}\left(\frac{1}{m}\right)\right)\left(\frac{3}{2}\right)^{3-k}\frac{2}{9\sqrt{\pi m}}\left(\frac{3}{2}\right)^{3r_0}
  =\left(1{+}\textup{O}\left(\frac{1}{m}\right)\right)\frac{3}{\sqrt{\pi m}}\left(\frac{3}{2}\right)^{m}.
  \end{align*}
  Finally, since $1+\textup{O}\left(\frac{1}{m}\right)\to1$
  as $m\to\infty$, the limit in~\eqref{E:A2} follows.
\end{proof}

\end{document}